\documentclass[twoside]{article}

\usepackage[accepted]{aistats2022}
\usepackage[sort]{natbib}

\usepackage{microtype}

\usepackage{url}
\usepackage{hyperref}%

\usepackage{blindtext}

\usepackage{amsmath}
\usepackage{amssymb}
\usepackage{mathtools}

\newcommand{\Ebb}{\mathbb{E}}
\newcommand{\Cbb}{\mathbb{C}}
\newcommand{\Rbb}{\mathbb{R}}

\newcommand{\Tbb}{\mathbb{T}}
\newcommand{\Nbb}{\mathbb{N}}
\newcommand{\Xbb}{{\mathbb{X}}}

\newcommand{\onebb}{\mathbf{1}}

\newcommand{\Ncal}{\mathcal{N}}

\newcommand{\Bcal}{\mathcal{B}}

\newcommand{\Dcal}{\mathcal{D}}

\newcommand{\Ical}{\mathcal{I}}

\newcommand{\pderiv}[2]{\frac{\partial #1}{\partial #2}}

\newcommand{\diff}{\,\textup{d}}

\DeclareMathOperator{\gp}{\mathcal{GP}}

\DeclareMathOperator{\loc}{loc}

\usepackage{xcolor}

\usepackage{amsthm}

\usepackage[capitalise]{cleveref}

\newtheorem{proposition}{Proposition}
\newtheorem{lemma}[proposition]{Lemma}
\newtheorem{assumption}[proposition]{Assumption}

\crefname{assumption}{Assumption}{Assumptions}
\crefname{appendix}{Supplement}{Supplements}

\usepackage{xr}

\crefname{supp}{Supplement}{Supplements}

\usepackage{tikz}
\usetikzlibrary{bayesnet,arrows}
\tikzset{>=stealth'}

\begin{document}

%
\runningtitle{Probabilistic Numerical Method of Lines}

%

\twocolumn[

  \aistatstitle{Probabilistic Numerical Method of Lines \\
    for Time-Dependent Partial Differential Equations
  }

  \aistatsauthor{ Nicholas Krämer \And Jonathan Schmidt \And Philipp Hennig }

  \aistatsaddress{ University of Tübingen \\Tübingen, Germany \And  University of Tübingen \\Tübingen, Germany \And University of Tübingen and \\Max Planck Institute for Intelligent Systems \\Tübingen, Germany } ]

\begin{abstract}
  This work develops a class of probabilistic algorithms for the numerical solution of nonlinear, time-dependent partial differential equations (PDEs).
  Current state-of-the-art PDE solvers treat the space- and time-dimensions separately, serially, and with black-box algorithms, which obscures the interactions between spatial and temporal approximation errors and misguides the quantification of the \emph{overall} error.
  To fix this issue, we introduce a probabilistic version of a technique called \emph{method of lines}.
  The proposed algorithm begins with a Gaussian process interpretation of finite difference methods, which then interacts naturally with filtering-based probabilistic ordinary differential equation (ODE) solvers because they share a common language: Bayesian inference.
  Joint quantification of space- and time-uncertainty becomes possible {without} losing the performance benefits of well-tuned ODE solvers.
  Thereby, we extend the toolbox of probabilistic programs for differential equation simulation to PDEs.
\end{abstract}

\section{INTRODUCTION}

This work develops a class of probabilistic numerical (PN) algorithms for the solution of initial value problems based on {partial} differential equations (PDEs).
PDEs are a widely-used way of modelling physical interdependencies between temporal and spatial variables.
With the recent advent of physics-informed neural networks \citep{raissi2019physics}, neural operators \citep{lu2019deeponet,li2021fourier}, and neural ordinary/partial differential equations \citep{chen2018neural,gelbrecht2021neural}, PDEs have rapidly gained popularity in the machine learning community, too.
Let $F$, $h$, and $g$ be given nonlinear functions. Let the domain $\Omega \subset \Rbb^d$ with boundary $\partial \Omega$ be sufficiently well-behaved so that the PDE (below) has a unique solution.\footnote{One common assumption is that $\Omega$ must be open and bounded, and that $\partial \Omega$ must be differentiable. But requirements vary across differential equations \citep{evans2010partial}.}
$\Dcal$ shall be a differential operator.
The reader may think of the Laplacian, $\Dcal = \sum_{i=1}^d \pderiv{^2}{x_i^2}$, but the algorithm is not restricted to this case.
The goal is to approximate an unknown function $u: \Omega \rightarrow \Rbb^L$ that solves
\begin{align}\label{eq:pde}
  \pderiv{}{t}u(t, x)
   & = F(t, x, u(t, x), \Dcal u(t, x)),
\end{align}
for $t \in [t_0, t_\text{max}]$ and $x \in \Omega$,
subject to initial condition
\begin{align}
  u(t_0, x) & = h(x), \quad x \in \Omega,
\end{align}
and boundary conditions
$  \Bcal u(t, x) = g(x)$, $x \in \partial \Omega$. The differential operator $\Bcal$ is usually the identity (Dirichlet conditions) or the derivative along normal coordinates  (Neumann conditions).
Except for only a few problems, PDEs do not admit closed-form solutions, and numerical approximations become necessary.
This also affects machine learning strategies such as physics-informed neural networks or neural operators, for example, because they rely on fast generation of training data.

\begin{figure*}[t!]
  \centering
  \includegraphics{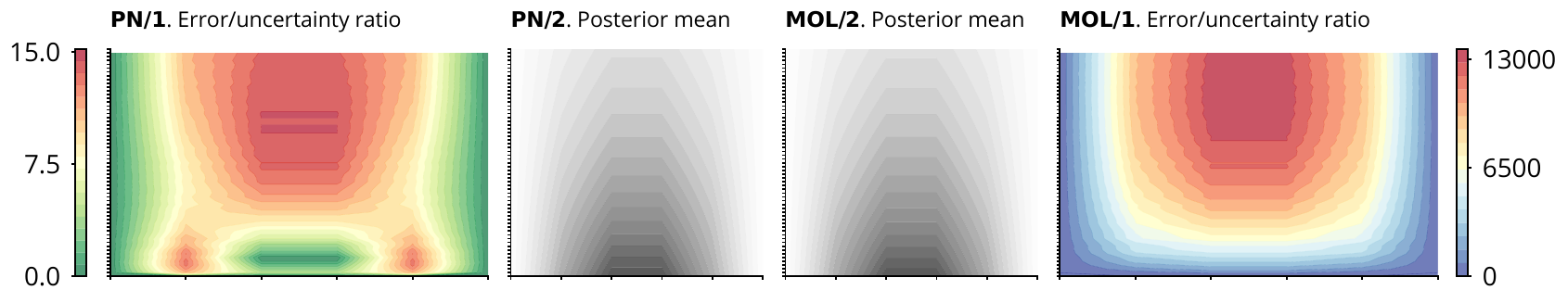}
  \caption{\textit{PNMOL fixes bad calibration of the MOL/ODE filter combination:}
    Posterior means and error/uncertainty ratios of PNMOL (left) and MOL with a conventional PN solver (right) on a fine time grid ($y$-axis) and a coarse space grid ($x$-axis) for the heat equation.
    The means are indistinguishable (*/2).
    PN w/ MOL is poorly calibrated (error/uncertainty ratios $\sim 10^5$; */1), but PNMOL acknowledges all inaccuracies.
  }
  \label{fig:figure1}
\end{figure*}
One common strategy for solving PDEs, called the \emph{method of lines} \citep[MOL;][]{schiesser2012numerical}, first discretises the spatial domain $\Omega$ with a grid $\Xbb := (x_0, ..., x_N)$, and then uses this grid to approximate the differential operator $\Dcal$ with a matrix-vector product
\begin{align}\label{eq:teaser_differentiation_matrix}
  (\Dcal u)(t, \Xbb) \approx D u(t, \Xbb), \quad D \in \Rbb^{(N+1) \times (N+1)},
\end{align}
where we use the notation $u(t, \Xbb) = (u(t, x_n))_{n=0}^N$.
Replacing the differential operator $\Dcal$ with the matrix $D$ turns the PDE into a system of ordinary differential equations (ODEs; more information about $D$ follows).
Standard ODE solvers can then numerically solve the resulting initial value problem.

This approach has one central problem.
Discretising the spatial domain, and only then applying an ODE solver, turns the PDE solver into a \emph{pipeline of two numerical algorithms} instead of a single algorithm.
This is bad because as a result of this serialisation, the error estimates returned by the ODE solver are unreliable.
The solver lacks crucial information about whether the spatial grid consists of, say, $N=4$ or $N=10^7$ points.
It is intuitive that a coarse spatial grid puts a lower bound on the overall precision, even if the ODE solver uses small time steps. But since this is not ``known'' to the ODE solver, not even to a probabilistic one, it may waste computational resources by needlessly decreasing its step-size and may deliver (severely) overconfident uncertainty estimates (for example, \cref{fig:figure1}).
This phenomenon will be confirmed by the experiments in \cref{sec:experiments}.
Such overconfident uncertainty estimates are essentially useless to a practitioner, which complicates the usage of method of lines approaches within probabilistic simulation of differential equations; for example, in the context of latent force inference or inverse problems.

The present work provides a solution to this problem by revealing a way of tracking spatiotemporal correlations in the PDE solutions while preserving the computational advantages of traditional MOL through a probabilistic numerical solver.
The main idea of the present paper is the following.
The error that is introduced by approximating the differential operator $\Dcal$ with a matrix $D$ can be quantified probabilistically, and its time-evolution acts as a latent force on the resulting ODE.
The combined posterior over the PDE solution and the latent force can be computed efficiently with a PN ODE solver.
More specifically, we contribute:

\paragraph{Probabilistic discretisation}
\cref{sec:probabilistic_discretisation,sec:localisation} present a probabilistic technique for constructing finite-dimensional approximations of $\Dcal$ that quantify the discretisation error.
It translates an unsymmetric collocation method \citep{kansa1990multiquadrics} to the language of Gaussian processes and admits efficient, sparse approximations akin to radial-basis-function-generated finite-difference formulas  \citep{tolstykh2003using}.

\paragraph{Calibrated PDE filter}
Building on the probabilistic discretisation,
\cref{sec:pn_mol} explains how the discretisation and the ODE solver need not be treated as two separate entities.
By applying the idea behind the method of lines to a \emph{global} Gaussian process prior, it becomes evident how the discretisation uncertainty naturally evolves over time.
With a filtering-based probabilistic ODE solver (``ODE filter''), posteriors over the latent error process and the PDE solution can be inferred in {linear time}, and without ignoring spatial uncertainties like traditional MOL algorithms do.

\cref{sec:hyper_parameters} explains hyperparameter choices and calibration.
\cref{sec:related_work} discusses connections to existing work.
\cref{sec:experiments} demonstrates the advantages of the approach over conventional MOL.
Altogether, PNMOL enriches the toolbox of probabilistic simulations of differential equations by a calibrated and efficient PDE solver.

\section{PN DISCRETISATION}
\label{sec:probabilistic_discretisation}

This section derives how to approximate a differential operator $\Dcal$ by a matrix $D$.
Let $u_x(z) \sim \gp(0,\gamma^2  k_x)$ be a Gaussian process on $\Omega$ with covariance kernel $k_x(z, z')$ and output scale $\gamma > 0$.
The presentation in \cref{sec:pn_mol} will assume a Gaussian process prior $u(t,x) \sim \gp(0, \gamma^2  k_t\otimes k_x)$, thus the notation ``$u_x$'' motivates $u_x$ as ``the $x$-part of $u(t,x)$''.\footnote{Here, $u_x$ does not refer to the partial derivative $ \pderiv{}{x}u$, which it sometimes does in the PDE literature.}
The boundary operator $\Bcal$ is ignored in the remainder because it can be treated in the same way as $\Dcal$.
Complete formulas including $\Bcal$, are in \cref{sec:boundary-conditions}.

The objective is to approximate the PDE dynamics in a way that circumvents the differential operator $\Dcal$,
\begin{align}
  f(t, \Xbb, u_x(\Xbb)) \approx F(t, \Xbb, u_x(\Xbb), (\Dcal u_x)(\Xbb)),
\end{align}
that is, $\Dcal u_x$ disappears from \cref{eq:pde} because $f$ replaces $F$.
For linear differential operators $\Dcal$, this reduces to replacing $(\Dcal u_x)(\Xbb)$ with a matrix-vector product $D u_x(\Xbb)$, $D \in \Rbb^{(N+1)\times (N+1)}$ (recall \cref{eq:teaser_differentiation_matrix}), for example, by using the method proposed below.
$\Dcal$ must not be present in the discretised model, because otherwise, the PDE does not translate into a system of ODEs, and we cannot proceed with the method of lines.
But with an approximate derivative based on only matrix-vector arithmetic, the computational efficiency of ODE solvers can be exploited to infer an approximate PDE solution.

Since $u$ is a Gaussian process, applying the linear operator $\Dcal$ yields another Gaussian process $\Dcal u_x \sim \gp(0, \Dcal^2 k_x)$,
where we abbreviate $\Dcal^2 k_x(z, z')=\Dcal \Dcal^* k_x(z, z')$ and $\Dcal^*$ is the adjoint of $\Dcal$ (in the present context, this means applying $\Dcal$ to $z'$ instead of $z$).
Conditioning $u_x$ on realisations of $u_x(\Xbb)$, and then applying $\Dcal$, yields another Gaussian process with moments
\begin{subequations}\label{eq:conditioned_gp_moments_full}
  \label{eq:conditioned_gp_moments}
  \begin{align}
    \hat{m}(z)     & = W_\Xbb(z) u_x(\Xbb),                                              \\
    \hat{k}(z, z') & = (\Dcal^2 k_x)(z, z') - W_\Xbb(z) k_x(\Xbb, \Xbb) W_\Xbb(z')^\top, \\
    W_\Xbb(z)      & := (\Dcal k_x)(z, \Xbb) k_x(\Xbb,\Xbb)^{-1}.
  \end{align}
\end{subequations}
Let $\xi_x \sim \gp (0, \gamma^2 \hat{k})$.
$\xi_x$ is independent of $u_x$, since $\hat{k}$ only depends on $\Dcal$, $k_x$, and $\Xbb$, not on $u_x(\Xbb)$.
Then,
\begin{equation}
  W_\Xbb(z) u_x(\Xbb) + \xi_x(z) \sim \gp(0, \Dcal \Dcal^* k_x)
\end{equation}
follows, which implies
\begin{align}\label{eq:approximation_of_d}
  p(\Dcal u_x(\cdot)) = p(W_\Xbb(\cdot) u_x(\Xbb) + \xi(\cdot)).
\end{align}
In particular, when evaluated on the grid, we obtain the matrix-vector formulation of the differential operator
\begin{equation}\label{eq:sum_independent_gp}
  (\Dcal u_x)(\Xbb) = W_\Xbb(\Xbb) u_x(\Xbb) + \xi_x(\Xbb).
\end{equation}
\cref{eq:sum_independent_gp} yields $\Dcal u_x$ solely as a transformation of values of $u_x$ with known, additive Gaussian noise $\xi_x$.
Altogether, the above derivation identifies the \emph{differentiation matrix} $D$ and the \emph{error covariance} $E$
\begin{align}
  D := W_\Xbb(\Xbb), \quad E := \hat k(\Xbb, \Xbb). \label{eq:error_matrix_E}
\end{align}
If we know $u_x(\Xbb)$, we have an approximation to $(\Dcal u_x)(\Xbb)$ that is wrong with covariance $ \gamma^2 E$.
The matrix $E$ makes the approach probabilistic (and new).

There are two possible interpretations for $D$ (and $E$).
$D$ appears in a method for solving PDEs $\Dcal u_x = f$ with radial basis functions, called \emph{unsymmetric collocation}, or \emph{Kansa's method} \citep{kansa1990multiquadrics,hon2001unsymmetric,schaback2007convergence}.
A related method, called symmetric collocation \citep{fasshauer1997solving,fasshauer1999solving},
has been built on and translated into a Bayesian probabilistic numerical method by \citet{cockayne2017probabilistic}.
The following derivation shows that a translation is possible for the probabilistic discretisation, respectively unsymmetric collocation, as well:
\cref{eq:approximation_of_d} explains
\begin{subequations}
  \begin{align}
    (\Dcal u_x)(\Xbb) \mid u_x(\Xbb), \xi_x
     & \sim \delta\left[D u_x(\Xbb) + \xi_x \right] \\
     & = \Ncal(D u_x(\Xbb), \gamma^2 E).
  \end{align}
\end{subequations}
$\delta$ is the Dirac delta.
In other words, $D$ and $E$ (respectively $D$ and $\xi$) estimate $\Dcal u_x$ from values of $u_x$ at a set of grid-points -- just like finite difference formulas do.
By construction, it is a Bayesian probabilistic numerical algorithm in the technical sense of \citet{cockayne2019bayesian}.
\cref{sec:bpnm} contains a formal proof of this statement.

The limitations of the above probabilistic discretisation are twofold.
First, computation of the differentiation matrix $D$ involves inverting the kernel Gram matrix $k_x(\Xbb, \Xbb)$.
For large point sets, this matrix is notoriously ill-conditioned \citep{schaback1995error}.
Second, computation -- and application -- of $D$ is expensive because $D$ is a dense matrix with as many rows and columns as there are points in $\Xbb$.
Since loosely speaking, a derivative is a ``local'' phenomenon (other than e.g.~computing integrals), intuition suggests that $\Dcal u_x$ can be approximated more cheaply by exploiting this locality.
This thought leads to localised differentiation matrices and probabilistic numerical finite differences.

\section{PN FINITE DIFFERENCES}
\label{sec:localisation}

The central idea of an efficient approximation of the probabilistic discretisation is to approximate the derivative of $u_x$ at each point $x_n$ in $\Xbb$ individually.
This leads to a row-by-row assembly of $D$, which can be sparsified naturally and without losing the ability to quantify numerical discretisation/differentiation uncertainty.

\begin{figure*}[t!]
  \includegraphics{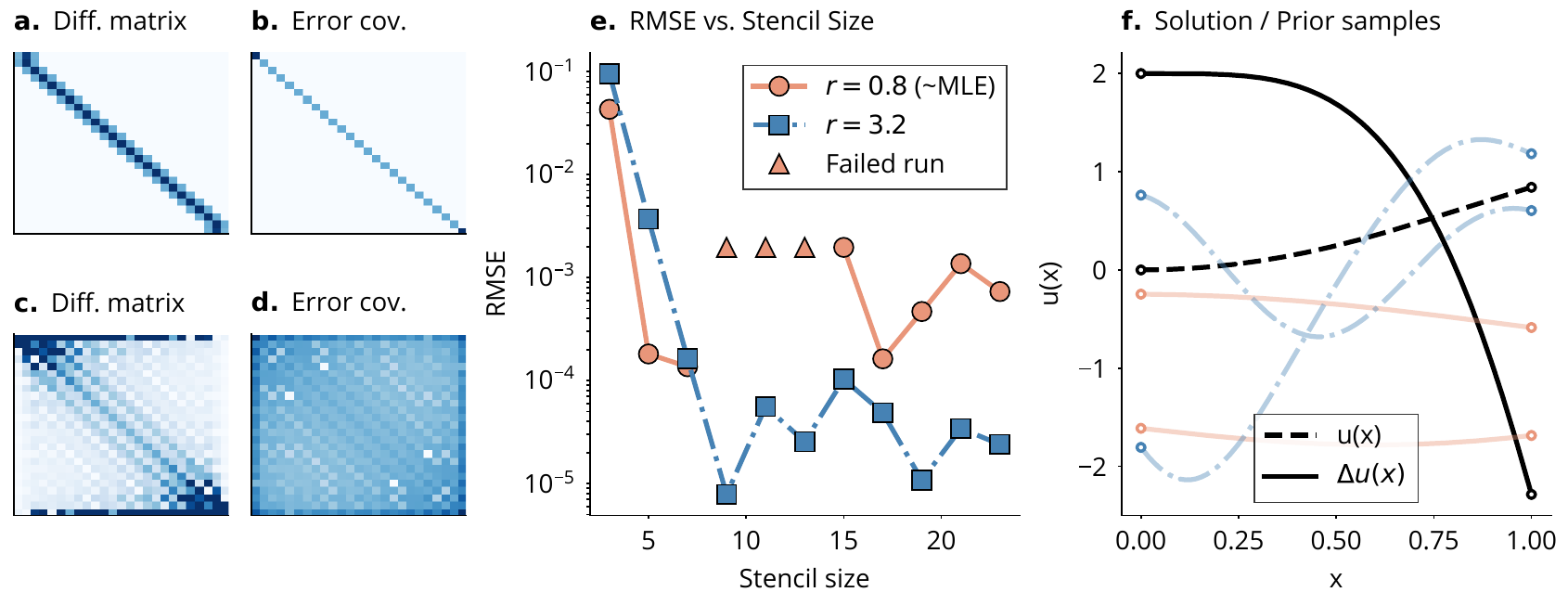}
  \caption{\textit{
      Discretise the Laplacian with a local and global approximation:}
    The target  is the Laplacian $\Dcal = \Delta$ of $u(x) = \sin(\|x\|^2)$ (f).
    \textit{Left:}
    Sparsity pattern of the differentiation matrix and error covariance matrix for the localised approximation (a, b) and the global approximation (c, d) on a mesh with $N=25$ points. The approximation is least certain at the boundaries.
    \textit{Centre:}
    The root-mean-square error between $\Delta u$ and its approximation decreases with an increased stencil size but the approximation breaks down for larger stencils (e), likely due to ill-conditioned kernel Gram matrices.
    A maximum likelihood estimate of the input scale $r \in \Rbb$ of the square exponential kernel $k(x,y) = e^{-r^2 \|x-y\|^2}$ based on data $u_x(x)$ alone does not necessarily lead to well-conditioned system matrices, nor does it generally imply a low RMSE (e).
    \textit{Right:}
    Samples from the prior GP $u_x$ for both length scales are shown next to the solution and the target function (f; the colours match the colours in the RMSE plot).
    \emph{Increasing stencil sizes improves the accuracy until stability concerns arise.}
  }
  \label{fig:stencil_size_comparison}
\end{figure*}
Write the matrix $D$ as a vertical stack of $N+1$ row-vectors, $D = (d_0, ..., d_N)^\top$,
and denote by $e_n$ the $n$th diagonal element of $E$, i.e. the variance of the approximation of $\Dcal u_x$ at the $n$th grid-point in $\Xbb$.
Thus,
\begin{align}\label{eq:scalarified_posterior}
  \Dcal u_x(x_n) \mid u_x(\Xbb), \xi_x \sim \delta[d_n^\top u_x(\Xbb) + \xi_{x, n}]
\end{align}
with the scalar random variable $\xi_{x, n} \sim \Ncal(0, \gamma^2 e_n)$.
If, in \cref{eq:scalarified_posterior}, we replace $u_x(\Xbb)$ by the values of $u_x$ at only a local neighbourhood of $x_n$ (a ``stencil''),
\begin{align}
  x_{\loc(n)}=\{x_{n-k},..., x_n, ..., x_{n+k}\},
\end{align}
the general form of \cref{eq:scalarified_posterior} is preserved, but $d_n$ consists of only $2k+1$ elements, with $k \ll N$,
\begin{align}\label{eq:differentiation_matrix_D_localised}
  d_n = (\Dcal k_x)(x_n, {x}_{\loc(n)}) k_x(x_{\loc(n)}, x_{\loc(n)})^{-1}
\end{align}
instead of $N+1$ elements.
The error $e_n$ becomes
\begin{align}\label{eq:error_matrix_E_localised}
  e_n & := (\Dcal^2 k_x)(x_n, x_n) - d_n k_x (x_{\loc(n)}, x_{\loc(n)})d_n^\top.
\end{align}
If the coefficients in the new $d_n$ are all embedded into $D$ according to the indices of $x_{\loc(n)}$ in $\Xbb$, $D$ becomes sparse and $E$ becomes diagonal.
More precisely, $D_{ij} \neq 0$ if and only if $j \in \loc(i)$.
On a one-dimensional domain $\Omega$, and with $k=1$, $D$ is a banded matrix with bandwidth 3.
Furthermore, due to the point-by-point construction, choosing $k= N/2$ implies that the original $D$ from \cref{sec:probabilistic_discretisation} is recovered; however, $E$ remains diagonal.
The sparse approximation resolves many of the performance issues that persist with the global approximation (\cref{fig:stencil_size_comparison}).

The sparsified differentiation matrix is known as the radial-basis-function-generated finite difference matrix \citep{driscoll2002interpolation,shu2003local,tolstykh2003using,fornberg2015primer}.
The correspondence to finite-difference formulas stems from the fact that if $k_x$ were a polynomial kernel and the mesh $\Xbb$ were equidistant and one-dimensional, the coefficients in $D$ would equal the standard finite difference coefficients \citep{fornberg1988generation}.
The advantage of the more general, kernel-based finite difference approximation over polynomial coefficients is a more robust approximation, especially for non-uniform grid points and in higher dimensions \citep{tolstykh2003using,fornberg2015primer}.
The Bayesian point of view does not only add uncertainty quantification in the form of $E$ but also reveals that a practitioner may choose suitable kernels $k_x$ to include prior information into the PDE simulation.

\section{PN METHOD OF LINES}
\label{sec:pn_mol}
The previous two sections have been concerned with approximating the derivative of a function, purely from observations of this function on a grid.
Next, we use these strategies to solve time-dependent PDEs.
To this end, we combine the probabilistic discretisation with an ODE filter. As opposed to non-probabilistic MOL, this combination quantifies the leak of information between the space discretisation and the ODE solution.
Simply put, what is commonly treated as a pipeline of disjoint solvers, becomes more of a single algorithm.

\subsection{Spatiotemporal Prior Process}
For any function $\varphi=\varphi(t)$, let $\vec{\varphi}$ be the stack of $\varphi$ and its first $\nu$ derivatives, $\vec{\varphi}(t) := (\varphi(t), \dot \varphi(t), ..., \varphi^{(\nu)}(t))$.
This abbreviation is convenient because some stochastic processes (like the integrated Wiener or Mat\'ern process) do not have the Markov property, but the stack of state and derivatives does.
The following assumption is integral for the probabilistic method of lines.

\begin{assumption}\label{assumption:separable_prior_covariance}
  Assume a Gaussian process prior with separable covariance structure,
  \begin{align}
    u=u(t,x) \sim \gp(0, \gamma^2 k_t \otimes k_x)
  \end{align}
  for some output-scale $\gamma > 0$. $k_t \otimes k_x$ is the product kernel $(k_t \otimes k_x)(t, t', x, x') = k_t(t, t')k_x(x, x')$.
\end{assumption}
Compared to the traditional method of lines, where the temporal and spatial dimensions are treated independently and with black-box methods, \cref{assumption:separable_prior_covariance} is mild: even though the covariance is separable, the algorithm still starts with a single, global Gaussian process.
\cref{assumption:separable_prior_covariance} allows choosing temporal kernels that eventually lead to a fast algorithm:
\begin{assumption}\label{assumption:markov_sde}
  Assume $k_t$ is a covariance kernel that allows conversion to a linear, time-invariant stochastic differential equation (SDE) in the following sense:
  For any $\upsilon \sim \gp(0, \gamma^2 k_t)$, $\vec{\upsilon}$ solves the SDE
  \begin{align}
    \diff \vec{\upsilon}(t) = A \vec{\upsilon}(t) \diff t + B \diff w(t)
  \end{align}
  subject to Gaussian initial conditions
  \begin{align}
    \vec{\upsilon}(t_0) \sim \Ncal(m_0, \gamma^2 C_0),
  \end{align}
  for $A$, $B$, $m_0$, and $C_0$ that derive from $k_t$, and for a one-dimensional Wiener process $w$ with diffusion $\gamma^2$.
\end{assumption}
\cref{assumption:markov_sde} is satisfied, for instance, for the integrated Wiener process or the Mat\'ern process; many more examples are given in Chapter 12 of the book by \citet{sarkka2019applied}.
\cref{assumption:separable_prior_covariance,assumption:markov_sde} unlock the machinery of probabilistic ODE solvers.

Next, we add spatial correlations into the prior SDE model.
The following \cref{lem:sde_representation_discretised} will simplify the derivations further below \citep{solin2016stochastic}.

\begin{lemma}\label{lem:sde_representation_discretised}
  Let $k_t$ be a covariance function that satisfies \cref{assumption:markov_sde}.
  Let $M \in \Rbb^{q\times q}$ be a matrix.
  For a process $\varphi \sim \gp(0, \gamma^2 k_t \otimes M)$, $\vec{\varphi}$ solves
  \begin{align}
    \diff \vec{\varphi}(t) = (A \otimes I) \vec{\varphi}(t) \diff t + (B \otimes I) \diff \hat w(t),
  \end{align}
  subject to Gaussian initial conditions
  \begin{align}
    \vec{\varphi}(t_0) \sim \Ncal(m_0 \otimes \onebb, \gamma^2 C_0 \otimes M),
  \end{align}
  where $\hat w$ is a $q$-dimensional Wiener process with constant diffusion $\gamma^2 M$, and $I \in \Rbb^{q\times q}$ is the identity.
  $m_0 \otimes \onebb$ is a vertical stack of $d$ copies of $m_0$.
\end{lemma}

\cref{lem:sde_representation_discretised} suggests that a spatiotemporal prior (\cref{assumption:separable_prior_covariance}) may be restricted to a spatial grid without losing the computational benefits that SDE priors provide.
Abbreviate $U(t) := u(t, \Xbb)$.
\cref{lem:sde_representation_discretised} gives rise to an SDE representation for $U(t)$ by choosing $M = k_x(\Xbb, \Xbb)$.
The same holds for the error model:
Recall the definition of the error covariance $E$ from \cref{eq:error_matrix_E} (or \cref{eq:error_matrix_E_localised} respectively, if the localised version is used).
$\xi \sim \gp(0, \gamma^2  k_t \otimes E)$ admits a state-space formulation due to \cref{lem:sde_representation_discretised}.
Through
\begin{subequations}
  \begin{align}
    D u(t, & \Xbb) +  \xi(t)                                                  \\
           & \sim \gp(0,\gamma^2  k_t \otimes [D k_x(\Xbb, \Xbb) D^\top + E]) \\
           & = \gp(0,\gamma^2  k_t \otimes (\Dcal \Dcal^* k_x)(\Xbb, \Xbb)])
  \end{align}
\end{subequations}
it is evident that $\xi$ is an appropriate prior model for the time-evolution of the spatial discretisation error.
This error being part of the probabilistic model is the advantage of PNMOL over non-probabilistic versions.

\subsection{Information Model}
The priors over $U$ and $\xi$ become a probabilistic numerical PDE solution by conditioning $U$ and $\xi$ on ``solving the PDE'' at a number of grid points as follows.
Recall from \cref{eq:sum_independent_gp} how $\Dcal u(t, \Xbb)$ is approximated by $D u(t, \Xbb) + \xi$.
The residual process
\begin{align}
  r(t) := \dot U(t) - F(t, \Xbb, U(t), D U(t) + \xi)
\end{align}
measures how well realisations of $\vec{U}$ and $\vec\xi$ solve the PDE.
$\dot U$ and $U$ are components in $\vec{U}$, therefore all operations are tractable given the extended state vectors $\vec{U}$ and $\vec{\xi}$.
Conditioning $\vec{U}$ and $\vec\xi$ on $r(t) = 0$ for all $t$ yields a probabilistic PDE solution. However, in practice, we need to discretise the time variable first in order to be able to compute the (approximate) posterior.

\subsection{Time-Discretisation}
Let $\Tbb := (t_0, ..., t_K)$ be a grid on $[t_0, t_\text{max}]$.
Define the time-steps $h_k := t_{k+1} - t_k$.
Restricted to the time grid,
\begin{subequations}\label{eq:discrete-time-transitions}
  \begin{align}
    \vec{U}(t_{k+1}) \mid \vec{U}(t_{k})     & \sim \Ncal(\Phi(h_k) \vec{U}(t_k), \gamma^2 \Sigma_U(h_k))    \\
    \vec{\xi}(t_{k+1}) \mid \vec{\xi}(t_{k}) & \sim \Ncal(\Phi(h_k) \vec{\xi}(t_k), \gamma^2\Sigma_\xi(h_k))
  \end{align}
\end{subequations}
with transition matrix
\begin{subequations}
  \begin{align}
    \Phi(h_k)   := \breve{\Phi}(h_k) \otimes I := \exp(A h_k)\otimes I
  \end{align}
\end{subequations}
and process noise covariance matrices
\begin{subequations}\label{eq:process-noise-covariances}
  \begin{align}
    \Sigma_U(h_k)       & := \breve{\Sigma}(h_k) \otimes k_x(\Xbb, \Xbb)                                         \\
    \Sigma_\xi(h_k)     & := \breve{\Sigma}(h_k) \otimes E                                                       \\
    \breve{\Sigma}(h_k) & := \int_0^{h_k} \breve{\Phi}(h_k-\tau)B B^\top  \breve{\Phi}(h_k-\tau)^\top\diff \tau.
  \end{align}
\end{subequations}
$\breve{\Phi}$ and $\breve{\Sigma}$ can be computed efficiently with matrix fractions \citep{sarkka2019applied}.
Integrated Wiener process priors, their time-discretisations, and practical considerations for implementation of high-orders are discussed in the probabilistic ODE solver literature \citep[e.g.][]{tronarp2021bayesian,kramer2020stable}.
On the discrete-time grid, the information model reads
\begin{align}\label{eq:discrete_likelihood}
  \begin{split}
    &r(t_k)  \mid \vec{U}(t_k), \vec{\xi}(t_k)                    \\
    & \sim \delta[\dot U(t_k) - F(t, \Xbb, U(t_k), D U(t_k) + \xi(t_k))].
  \end{split}
\end{align}
The complete setup is depicted in \cref{fig:pgm_mol}.
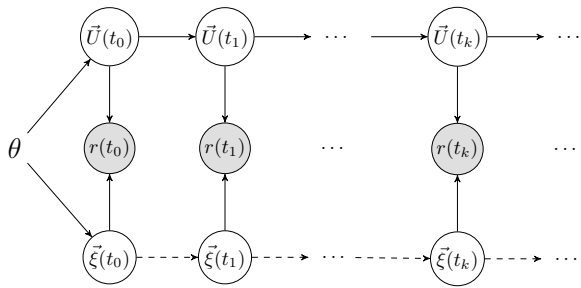
\begin{figure}[t!]
  \centering
  \scalebox{0.75}{
    \begin{tikzpicture}

      \node[latent] (u0) {\(\vec{U}(t_0)\)};%
      \node[latent, right = of u0] (u1) {\(\vec{U}(t_1)\)};%
      \node[const, right = of u1] (etc1) {~~\dots\quad~};%
      \node[latent, right = of etc1] (un) {\(\vec{U}(t_k)\)};%
      \node[const, right = of un] (etc2) {~~\dots\quad~};%

      \node[obs, below = of u0] (z0) {\(r(t_0)\)};%
      \node[obs, below = of u1] (z1) {\(r(t_1)\)};%
      \node[const, right = of z1] (etc5) {~~\dots~~};%
      \node[obs, below = of un] (zn) {\(r(t_k)\)};%
      \node[const, right = of zn] (etc6) {~~\dots~~};%

      \node[latent, below = of z0] (e0) {\(\vec{\xi}(t_0)\)};%
      \node[latent, below = of z1] (e1) {\(\vec{\xi}(t_1)\)};%
      \node[const, right = of e1] (etc3) {~~\dots~~};%
      \node[latent, below = of zn] (en) {\(\vec{\xi}(t_k)\)};%
      \node[const, right = of en] (etc4) {~~\dots~~};%

      \node[const, left = of z0, outer sep=3] (kX) {\Large $\theta$};%

      \edge {u0} {u1}
      \edge {u1} {etc1}
      \edge {etc1} {un}
      \edge {un} {etc2}

      \edge[dashed] {e0} {e1}
      \edge[dashed] {e1} {etc3}
      \edge[dashed] {etc3} {en}
      \edge[dashed] {en} {etc4}

      \edge{u0, e0}{z0}
      \edge{u1, e1}{z1}
      \edge{un, en}{zn}

      \edge{kX}{u0, e0}

    \end{tikzpicture}
  }
  \caption{\textit{Graphical visualisation of PNMOL:}
    The states $\vec{\xi}$ and $\vec{U}$ are conditionally independent given $\theta = (k_x, k_t, \Xbb, \Dcal, \gamma, \Tbb)$.
    \emph{The existence of the bottom row distinguishes PNMOL from other PDE solvers.}
    The dependencies between the $\xi(t_k)$, indicated by the dashed lines, are optional; removing them improves the efficiency of the algorithm; details are in \cref{sec:white-noise-approximation}.
  }
  \label{fig:pgm_mol}
\end{figure}

\subsection{Inference}
\label{subsec:inference}
Altogether, the probabilistic method of lines targets
\begin{align}\label{eq:probabilistic_pde_solution}
  p\left(\vec{U}, \vec{\xi} \,\left|\, [r(t_k) = 0]_{k=0}^K \right.\right)
\end{align}
which we will refer to as the \emph{probabilistic numerical PDE solution} of \cref{eq:pde}.
The precise form of the PN PDE solution is intractable because $F$ (and therefore the likelihood in \cref{eq:discrete_likelihood}) are nonlinear.
It can be approximated efficiently with techniques from extended Kalman filtering \citep{sarkka2019applied}, which is based on approximating the nonlinear PDE vector field $F$ with a first-order Taylor series. Inference in the linearised model is feasible with Kalman filtering and smoothing.
\citet{sarkka2019applied} summarise the details, and \citet{tronarp2019probabilistic} explain specific considerations for ODE solvers.

The role of $\xi$ and its impact on the posterior distribution (i.e.~the PDE solution) is comparable to that of a latent force in the ODE \citep{alvarez2009latent,hartikainen2012state,schmidt2021probabilistic}.
The coarser $\Xbb$ is, the larger is $E$, respectively $\xi$.
A large $\xi$ indicates how increasingly severe the misspecification of the ODE becomes.
Unlike traditional PDE solvers, inference according to the information model in \cref{eq:discrete_likelihood} embraces latent discretisation errors.
It does so by an algorithm that shares similarities with the latent force inference algorithm by \citet{schmidt2021probabilistic} (compare \cref{fig:pgm_mol} to \citet[Fig. 3]{schmidt2021probabilistic}), but the sources of misspecification are different in both works.

Inference of $\xi$ is expensive because the complexity of a Gaussian filter scales as $O(K d^3)$ where $d$ is the dimension of the state-space model.
In the present case, $d=2(N+1)(\nu+1)$ because for $N+1$ spatial grid points, PNMOL tracks $\nu$ time-derivatives of $U$ and $\xi$.
Both $U$ and $\xi$ are $(N+1)$-dimensional.
If the only purpose of $\xi$ is to incorporate a measure of spatial discretisation uncertainty into the information model of the PDE solver, tracking $\xi$ in the state space is not required, as long as we introduce another approximation.

\subsection{White-Noise Approximation}
\label{sec:white-noise-approximation}

Recall $\xi \sim \gp(0, \gamma^2 k_t \otimes E)$, i.e.~$\xi$ is an integrated Wiener process in time and a Gaussian $\Ncal(0,  \gamma^2 E)$ in space.
One may relax the temporal integrated Wiener process prior to a white noise process prior; that is,
$\xi(t)$ is independent of $\xi(s)$ for $s \neq t$, and $\xi(t) \sim \Ncal(0, \gamma^2 E)$ for all $t$.
The state-space realisation of a white noise process is trivial because there is no temporal evolution (recall the dashed lines in \cref{fig:pgm_mol}).
To understand the impact of $E$ in the white-noise formulation on the information model, consider the semi-linear PDE
\begin{align}\label{eq:semilinear_pde}
  F_\text{semi}(t, x, u, \Dcal u) := \Dcal u(t,x) + f(u(t,x))
\end{align}
for some nonlinear $f$.
\cref{eq:discrete_likelihood}  becomes
\begin{subequations}\label{eq:white_noise_information}
  \begin{align}
    r_\text{white}(t_k) & \mid \vec{U}(t_k) \sim \Ncal( \rho(\vec{U}(t)), \gamma^2 E) \\
    \rho(\vec{U}(t))    & := \dot U(t) - D U(t) - f(U(t))
  \end{align}
\end{subequations}
The Dirac likelihood in \cref{eq:discrete_likelihood} becomes a Gaussian with the measurement noise $E$.
Versions of \cref{eq:white_noise_information} for other nonlinearities are in
\cref{sec:white-noise}.

The advantage of the white-noise approximation over the latent-force version is that the state-space model is precisely half the size because $\xi$ is not a state variable anymore.
Since the complexity of PNMOL depends cubically on the dimension of the state-space, half as many state variables improve the complexity of the algorithm by a factor $2^3=8$.
Arguably, $\gamma^2 E$ entering the information model as a measurement covariance may also provide a more intuitive explanation of the impact that the statistical quantification of the discretisation error has on the PDE solution:
the larger $E$, the less strictly is a zero PDE residual enforced during inference.
The error covariance $E$ regularises the impact of the information model, and, in the limit of $E \rightarrow \infty$, yields the prior as a PDE solution.
Intuitively put, for $E>0$, the algorithm puts less trust in the residual information $r(\cdot)$ than for $E=0$.

\section{HYPERPARAMETERS}
\label{sec:hyper_parameters}

\paragraph{Kernels}
As common in the literature on probabilistic ODE solvers, we use temporal integrated Wiener process priors \citep[e.g.][]{tronarp2019probabilistic,bosch2021calibrated}.
In the experiments, the order of integration is $\nu \in \{1, 2 \}$.
Using low order ODE solvers is not unusual for MOL implementations \citep{cash1996mol}.
Spatial kernels need to be sufficiently differentiable to admit the formula in \cref{eq:error_matrix_E}.
In this work, we use squared exponential kernels.
Choosing their input scale is not straightforward (recall \cref{fig:stencil_size_comparison}; we found $r=0.25$ to work well across experiments).
Other sufficiently regular kernels, e.g. rational quadratic or Mat\'ern kernels, would work as well.
Polynomial kernels recover traditional finite difference weights \citep{fornberg1988generation}, but like for any other feature-based kernel, $k(x, y) = \Phi(x)^\top \Phi(y)$ holds, thus
both summands in \cref{eq:error_matrix_E} (and in \cref{eq:error_matrix_E_localised}) cancel out. The discretisation uncertainty $E$ would be zero.
In this case, PNMOL gains similarity to the traditional method of lines combined with a probabilistic ODE solver. (The boundary conditions would be treated slightly differently; we refer to \cref{sec:boundary-conditions}.)

\paragraph{Spatial Grid}
The spatial grid can be any set of freely scattered points.
Spatial neighbourhoods can, for example, be queried from a KD tree  \citep{bentley1979algorithms}.
For the simulations in the present paper, we use equispaced grids.
PNMOL's requirements on the grid differ from, for instance, finite element methods, in that there needs to be no notion of connectivity or triangulation between the grid points.
PN finite differences only require stencils; in light of the stability results in \cref{fig:stencil_size_comparison}, we choose them maximally small.

\paragraph{Output-Scale}
The output-scale $\gamma$ calibrates the width of the posterior and can be tuned with quasi-maximum likelihood estimation.
Omitting the boundary conditions, this means (recall $r$ from \cref{eq:discrete_likelihood})
\begin{align}\label{eq:quasi_mle_output_scale}
  (\hat\gamma)^2 := \frac{1}{(N+1)(K+1)} \sum_{k=0}^K \| \Ebb[r(t_k)]\|_{\Cbb[r(t_k)]^{-1}}^2,
\end{align}
where the mean $\Ebb[r(t_k)]$ and the covariance between the output-dimensions of $r(\cdot)$ at time $t_k$, $\Cbb[r(t_k)]$, emerge from the same Gaussian approximation that computes the approximate posterior (\cref{sec:output-scale-calibration}).
\cref{eq:quasi_mle_output_scale} uses the Mahalanobis norm $\|x\|_{A}^2 = x^\top A x$.

\begin{figure*}[t!]
  \centering
  \includegraphics{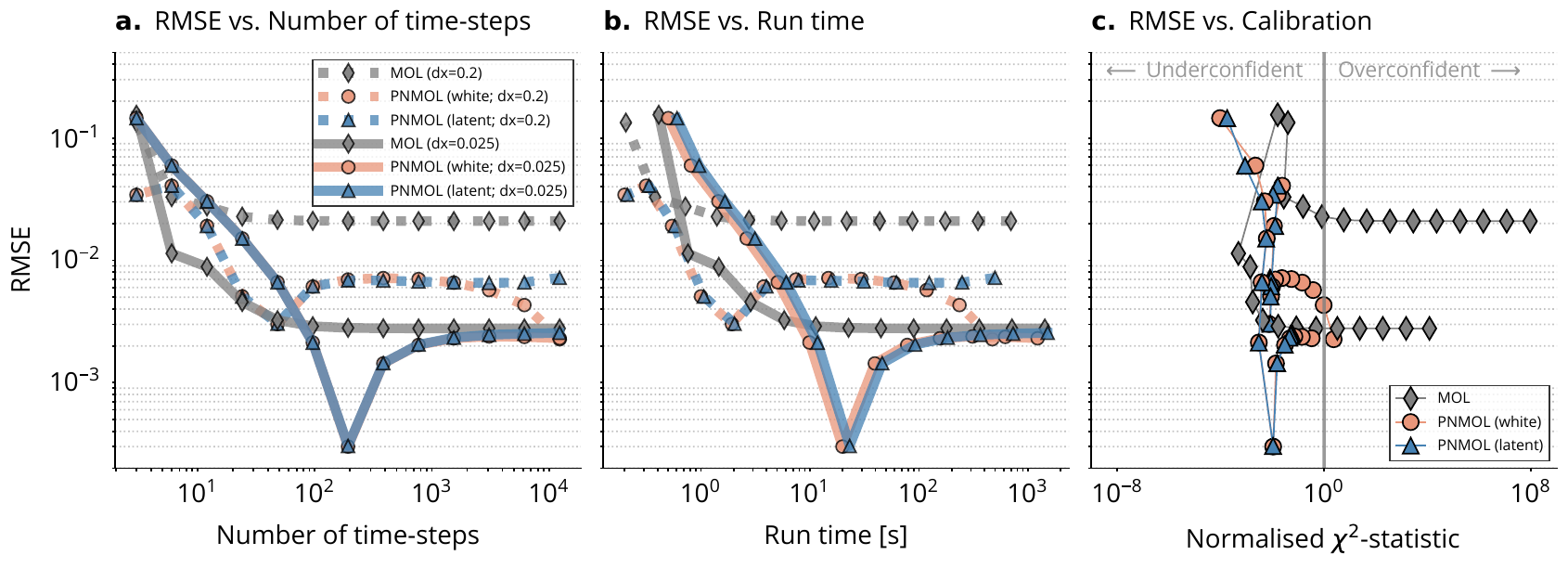}
  \caption{\textit{Quantify the global error:}
    Work vs. precision vs. calibration of PNMOL in the latent-force version (blue) and the white-noise version (orange), compared to a traditional PN ODE solver combined with conventional MOL (grey), on the spatial Lotka-Volterra model.
    Two kinds of curves are shown: one for a coarse (dotted), and one for a fine spatial mesh (solid).
    A reference is computed by discretising the spatial domain with a ten times finer mesh and solving the ODE with backward differentiation formulas.
    The RMSE of both methods stagnates once a certain accuracy is reached, but PNMOL appears to reach a slightly lower RMSE for $dx=0.2$ (left, middle; perhaps due to different treatment of boundary conditions; \cref{sec:boundary-conditions}).
    The run time of PNMOL-white is comparable to that of MOL, and the run time of PNMOL-latent is slightly longer (middle).
    The calibration of PNMOL, measured in the normalised $\chi^2$-statistic of the Gaussian posterior (so that the ``optimum'' is 1, not $d$), remains close to $1$ but is slightly underconfident.
    With decreasing time-steps, MOL is poorly calibrated.
  }
  \label{fig:calibration-lotka-volterra}
\end{figure*}

\section{RELATED WORK}
\label{sec:related_work}

Connections to non-probabilistic numerical approximation have been discussed in \cref{sec:probabilistic_discretisation,sec:localisation}.
Recall from there that the strongest connections are to unsymmetric collocation \citep{kansa1990multiquadrics,hon2001unsymmetric,schaback2007convergence}, radial-basis-function-generated finite differences \citep{driscoll2002interpolation,shu2003local,tolstykh2003using,fornberg2015primer}, and collocation methods in general. We refer to \citet{fasshauer2007meshfree,fornberg2015primer} for a more comprehensive overview.
The literature on the method of lines is covered by, e.g., \citet{schiesser2012numerical}.
\citet{dereli2013meshless,hon2014meshless} combine collocation with MOL.
None of the above exploits the correlations between spatial and temporal errors.
The significance of estimating the interplay of both error sources for MOL has been recognised by \citet{berzins1988global,lawson1991balancing,berzins1991towards}.

\citet{cockayne2017probabilistic,owhadi2015bayesian,owhadi2017multigrid,raissi2017machine,raissi2018numerical} describe a probabilistic solver for PDEs relating to symmetric collocation approaches from numerical analysis. \citet{chen2021solving} extend the ideas to non-linear PDEs.
\citet{wang2021bayesian} continue the work of \citet{chkrebtii2016bayesian} in constructing an ODE/PDE initial value problem solver that uses (approximate) conjugate Gaussian updating at each time-step.
\citet{duffin2021statistical} solve time-dependent PDEs by discretising the spatial domain with finite elements, and applying ensemble and extended Kalman filtering in time. They build on the paper by \citet{girolami2020statistical}.
\citet{conrad2017statistical,abdulle2021probabilistic} compute probabilistic PDE solutions by randomly perturbing non-probabilistic solvers.
All of the above discard the uncertainty associated with discretising $\Dcal \approx D$.
Some papers achieve ODE-solver-like complexity for time-dependent problems \citep{wang2021bayesian,chkrebtii2016bayesian,duffin2021statistical}, while others compute a continuous-time posterior \citep{cockayne2017probabilistic,owhadi2015bayesian,owhadi2017multigrid,raissi2017machine,raissi2018numerical,chen2021solving}.
PNMOL does both.

The efficiency of the PDE filter builds on recent work on filtering-based probabilistic ODE solvers \citep{schober2019probabilistic,tronarp2019probabilistic,kersting2020convergence,kersting2016active,bosch2021calibrated,kramer2020stable,tronarp2021bayesian} and their applications \citep{kersting2020differentiable,schmidt2021probabilistic}.
\citet{frank2020probabilistic} apply an ODE filter to solve discretised PDEs.
Similar algorithms have been developed for other types of ODEs \citep{hennig2014probabilistic,john2019goode,kramer2021linear}.
The papers by \citet{chkrebtii2016bayesian,conrad2017statistical,abdulle2020random}, prominently feature ODEs.

\section{EXPERIMENTS}
\label{sec:experiments}

\paragraph{Code}
The code for the experiments\footnote{\url{https://github.com/schmidtjonathan/pnmol-experiments}}, and an implementation of the PN discretisation and PN finite differences described in \Cref{sec:probabilistic_discretisation,sec:localisation}\footnote{\url{https://github.com/pnkraemer/probfindiff}} are on GitHub.

\paragraph{Quantify The Global Error}
We investigate how the PN method of lines impacts numerical uncertainty quantification.
As a first experiment, we solve a spatial Lotka-Volterra model \citep{holmes1994spatial}, i.e. nonlinear predator-prey dynamics with spatial diffusion, on a range of temporal and spatial resolutions.
From the results in \cref{fig:calibration-lotka-volterra}, it is evident how the spatial accuracy limits the overall accuracy.
But also how traditional ODE filters combined with MOL fail to quantify numerical uncertainty reliably.
At any parameter configuration, it is \emph{either} the spatial or the temporal discretisation that dominates the error.
Decreasing the time-step alone lets the error not only stagnate but worsens the calibration because the ODE solver does not know how bad the spatial approximation is.

\paragraph{Which Error Dominates?}
To further examine which one of \emph{either} $\Delta x$ or $\Delta t$ dominates the approximation, we consider a second example: a spatial SIR model \citep{gai2020sir}.
We investigate more formally how increasing either, the time-resolution vs. the space-resolution, leads to a low overall error.
The results are in \cref{fig:work-vs-precision-sir}, and confirm the findings from \cref{fig:calibration-lotka-volterra} above.
\begin{figure}[t!]
  \centering
  \includegraphics{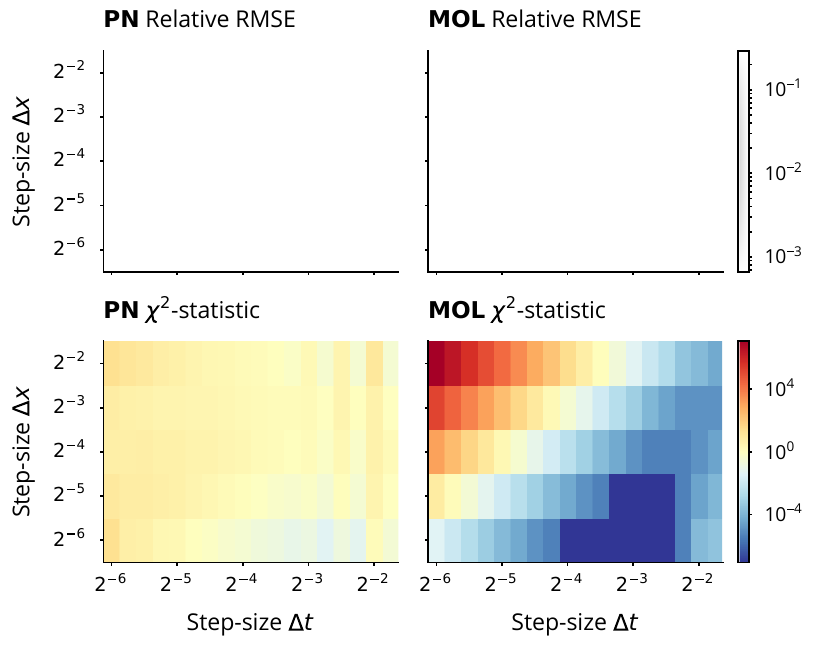}
  \caption{\textit{Which error dominates?}
    The relative RMSE is only small if \emph{both} $\Delta t$ and $\Delta x$ are small, which affects all solvers: PNMOL (top left) as well as traditional MOL combined with PN ODE solvers (top right).
    PN w/ MOL is severely overconfident for large $\Delta x$ and small $\Delta t$ (bottom right), while PNMOL delivers a calibrated posterior distribution (bottom left).
  }
  \label{fig:work-vs-precision-sir}
\end{figure}
Traditional PN ODE solvers with conventional MOL are unaware of the true, global approximation error.
PNMOL is not, despite being equally accurate.

\section{CONCLUSION}
We presented probabilistic strategies for discretising PDEs, and for making use of the resulting quantification of spatial discretisation uncertainty in a probabilistic ODE solver.
We discussed practical considerations, including sparsification of the differentiation matrices, keeping the dimensionality of the state-space low, and hyperparameter choices.
Altogether, and unlike traditional PDE solvers, the probabilistic method of lines unlocks quantification of spatiotemporal correlations in an approximate PDE solution, all while preserving the efficiency of adaptive ODE solvers. This makes it a valuable algorithm in the toolboxes of probabilistic programs and differential equation solvers and may serve as a backbone for latent force models, inverse problems, and differential-equation-centric machine learning.

\subsection*{Acknowledgements}

The authors gratefully acknowledge financial support
by the German Federal Ministry of Education and
Research (BMBF) through Project ADIMEM (FKZ
01IS18052B).
They also gratefully acknowledge finan-
cial support by the European Research Council through
ERC StG Action 757275 / PANAMA; the DFG Cluster
of Excellence ``Machine Learning - New Perspectives
for Science'', EXC 2064/1, project number 390727645;
the German Federal Ministry of Education and Re-
search (BMBF) through the Tübingen AI Center
(FKZ: 01IS18039A); and funds from the Ministry of
Science, Research and Arts of the State of Baden-
Württemberg.
Moreover, the authors thank the Inter-
national Max Planck Research School for Intelligent
Systems (IMPRS-IS) for supporting Nicholas Krämer.

The authors thank Nathanael Bosch for helpful feed-
back on the manuscript.
\medskip

\bibliography{bibfile}


\clearpage
\appendix

\thispagestyle{empty}

\onecolumn \makesupplementtitle

This supplement explains details regarding:
(i) treatment of boundary conditions (\cref{sec:boundary-conditions});
(ii) the probabilistic numerical discretisation from \cref{sec:probabilistic_discretisation} being a Bayesian probabilistic numerical method (\cref{sec:bpnm});
(iii) quasi-maximum-likelihood-estimation of the output-scale $\gamma$ (\cref{sec:output-scale-calibration});
and (iv) inference schemes for fully non-linear systems of partial differential equations (\cref{sec:white-noise}).

Recall the abbreviations from the main paper:
partial differential equation (PDE),
ordinary differential equation (ODE),
probabilistic numerics (PN),
method of lines (MOL),
probabilistic numerical method of lines (PNMOL).

\section{BOUNDARY CONDITIONS}
\label{sec:boundary-conditions}

\subsection{Setup}
In this section, we explain how PNMOL treats boundary conditions.
Recall the differential operator $\Dcal$ from \cref{eq:pde}, the differentiation matrix $D$ and the error covariance matrix $E$ from \cref{eq:error_matrix_E}, respectively \cref{eq:differentiation_matrix_D_localised,eq:error_matrix_E_localised}.
Also recall from \cref{subsec:inference} that in the ODE solve, each pair of states $(\vec{U}, \vec\xi)$ is conditioned on a zero PDE residual
\begin{align}\label{eq:information-operator-supplement}
  r(t) := \dot U(t) - F(t, \Xbb, U(t), D U(t) + \xi) \overset{!}{=} 0
\end{align}
where $U$, $\dot U$, and $\xi$ are components of the extended state vectors $\vec{U}$ and $\xi$.
The remainder of this section extends this framework to include boundary conditions $\Bcal u(t,x) = g(x)$ for all $x \in \partial \Omega$ and some function $g$ (\cref{eq:pde}).

\subsection{Discretised Boundary Conditions}
To augment the PDE residual information with boundary conditions, we begin by discretising $\Bcal$ probabilistically; either with global collocation as in \cref{sec:probabilistic_discretisation} or with PN finite differences as in \cref{sec:localisation}.
Below, we show the former, because the latter becomes accessible with the same modifications from \cref{sec:localisation}.
Let $u_x \sim \gp(0, \gamma^2 k_x)$.
Defining the differentiation matrix $B$ and the error covariance $R$,
\begin{align}
  B                  := (\Bcal k_x) (\Xbb, \Xbb) k_x(\Xbb, \Xbb)^{-1},        \quad
  R                  := (\Bcal^2 k_x)(\Xbb, \Xbb) - B k_x(\Xbb, \Xbb) B^\top,
\end{align}
we have access to boundary conditions:
For any $y \in \Rbb^{N+1}$,
\begin{align}\label{eq:bcond_as_conditioning}
  p((\Bcal u_x) (\Xbb) \mid u_x(\Xbb) = y) \sim \Ncal(B y, \gamma^2 R)
\end{align}
holds.
For Dirichlet boundary conditions, $\Bcal$ is the identity, thus $B = I_{N+1}$ is the identity matrix and $R \equiv 0$ is the zero matrix.
For Neumann conditions, $B$ is the discretised derivative along normal coordinates, and $R$ is generally nonzero (recall the explanation in \cref{sec:hyper_parameters} of the cases in which the error matrix is zero).
The derivation surrounding \cref{eq:bcond_as_conditioning} above suggests how the present framework would deal with boundary conditions that are subject to additive Gaussian noise.

\subsection{Latent Force}
Define the latent force $\vartheta=\vartheta(t) \sim \gp(0,\gamma^2  k_t \otimes R)$, which will play a role similar to $\xi$ but for the boundary conditions.
$\vartheta$ inherits a stochastic differential equation formulation from $k_t$ just like $\xi$ does (recall \cref{lem:sde_representation_discretised}).
Denote the stack of $\vartheta$ and its first $\nu \in \Nbb$ time-derivatives by $\vec{\vartheta}$.
Let $\Xbb_B \subset \Xbb$ be the subset of boundary points in $\Xbb$.
The full information operator (i.e.\ an extended version of \cref{eq:information-operator-supplement}) includes $\left.\Bcal u(t, x)\right|_{\partial \Omega} = g(x)$ as
\begin{align}
  r_\text{full}(t) :=
  \begin{pmatrix}
    \dot U(t) - F(t, \Xbb, U(t), D U(t) + \xi) \\
    \Bcal U(t) - g(\Xbb_B) - \vartheta(t)
  \end{pmatrix}
  \overset{!}{=} 0.
\end{align}
$r_\text{full}$ depends on $\vec{U}$, $\vec{\xi}$, and $\vec{\vartheta}$.
Conditioning $(\vec{U}, \vec{\xi}, \vec{\vartheta})$ on $r_\text{full}(t) = 0$ yields a probabilistic PDE solution that knows boundary conditions.
Notably, the bottom row in $r_\text{full}$ is linear in  $\left(\vec{U}, \vec{\xi}, \vec{\vartheta}\right)$ so there is no linearisation required to enable (approximate) inference.

\subsection{Comparison To MOL}
Boundary conditions in PNMOL enter through the information operator, i.e.\ on the same level as the PDE vector field $F$, which is different to conventional MOL:
In MOL, one only tracks the state variables in the interior of $\Omega$, i.e. $u(t, \Xbb \backslash \Xbb_B)$ because boundary conditions can be inferred from the interior straightforwardly.
For Dirichlet conditions, the boundary values are always dictated by $g(\Xbb_B)$.
Let $\pderiv{}{n}$ be the directional derivative taken in the direction normal to the boundary $\partial \Omega$.
For Neumann conditions, for some small $\lambda > 0$, we can approximate
\begin{align}
  \pderiv{}{n}u(t, x) = \frac{u(t, x) - u(t, x-\lambda n)}{\lambda} \overset{!}{=}g(x).
\end{align}
This viewpoint suggests $u(t,x) = \lambda g(x) + u(t, x-\lambda n)$ and $\lambda$ can be chosen such that $x - \lambda n$ is the nearest neighbor of $x \in \Xbb$.
While tracking only the state values in the interior of the domain has the advantage that the ODE system emerging from the method of lines is smaller than for PNMOL (which perhaps explains why in \cref{fig:calibration-lotka-volterra}, PNMOL achieves a lower error than MOL), MOL has two disadvantages:
(i) non-deterministic boundary conditions are not straightforward to include;
(ii) the finite difference approximation of Neumann conditions introduces errors.
PNMOL does not face the first issue and quantifies the error mentioned by the second issue.

\section{DISCRETISATION AS A BAYESIAN PROBABILISTIC NUMERICAL METHOD}
\label{sec:bpnm}

\citet{cockayne2019bayesian} introduce a formal definition of Bayesian and non-Bayesian probabilistic numerical methods.
More specifically, they define a probabilistic numerical method to be a tuple of an \emph{information operator} and a \emph{belief update operator}.
A PN method then becomes Bayesian if its output is the pushforward of a specific conditional distribution through the \emph{quantity of interest}.
All of those objects can be derived for the PN discretisation of $\Dcal$.
The same is true for the boundary conditions explained in \cref{sec:boundary-conditions} but is omitted in the following.
The proof of the statement below mirrors the explanation why Bayesian quadrature is a Bayesian PN method in Section 2.2 of the paper by \citet{cockayne2019bayesian}.
\begin{proposition}
  The approximation of $(\Dcal u_x)(\Xbb)$ with $D$ and $E$ as in \cref{eq:approximation_of_d,eq:error_matrix_E}, using evaluations of $u_x$ at a grid, is a Bayesian probabilistic numerical method.
\end{proposition}
\begin{proof}
  The prior measure is the GP prior $\gp(0, \gamma^2 k_x)$ and defined over some separable Banach space of sufficiently differentiable, real-valued functions.
  The information operator $\Ical[\varphi] := \varphi(\Xbb)$ evaluates a function $\varphi$ at the grid $\Xbb$.
  The belief update restricts the prior measure to the set of functions that interpolate $\varphi(\Xbb)$ (this is standard Gaussian process conditioning).
  The quantity of interest is the derivative $\Dcal u_x$, which results in the posterior distribution in \cref{eq:conditioned_gp_moments}.
\end{proof}

\section{QUASI-MAXIMUM-LIKELIHOOD ESTIMATION OF THE OUTPUT SCALE}
\label{sec:output-scale-calibration}

Proving the validity of the estimator in \cref{eq:quasi_mle_output_scale} parallels similar statements for similar settings \citep{tronarp2019probabilistic,bosch2021calibrated,kramer2021linear} and consists of two phases:
(i) showing that the posterior covariances are of the form $C_k = \gamma^2 \breve{C}_k$ for some $\breve{C}_k$, $k=0, ..., K$;
(ii) deriving the maximum likelihood estimators.

To show the first claim, recall the discrete-time transition of $U$ and $\xi$ from \cref{eq:discrete-time-transitions}
and the information model from \cref{eq:discrete_likelihood}.
Then, the fully discretised state-space model is
\begin{subequations}
  \begin{align}
    \vec{U}(t_0)                             & \sim \Ncal(m_0 \otimes \onebb, \gamma^2 C_0 \otimes k_x(\Xbb, \Xbb))      \\
    \vec{\xi}(t_0)                           & \sim \Ncal(m_0 \otimes \onebb, \gamma^2 C_0 \otimes E)                    \\
    \vec{U}(t_{k+1}) \mid \vec{U}(t_k)       & \sim \Ncal(\Phi(h_k) \vec{U}(t_k), \gamma^2\Sigma_U(h_k))                 \\
    \vec{\xi}(t_{k+1}) \mid \vec{\xi}(t_k)   & \sim \Ncal(\Phi(h_k) \vec{\xi}(t_k), \gamma^2\Sigma_\xi(h_k))             \\
    r(t_k) \mid \vec{U}(t_k), \vec{\xi}(t_k) & \sim \delta \left( \dot U(t_k) - F(t, \Xbb, U(t_k), D U(t_k) + \xi(t_k)).
    \right)
  \end{align}
\end{subequations}
All transitions have a process noise that depends multiplicatively on $\gamma^2$.
The Dirac likelihood is noise-free.
Therefore, the posterior covariance depends multiplicatively on $\gamma$ as well (which can be proved by an induction that is identical to the one in Appendix C of the paper by \citet{tronarp2019probabilistic}).

To show the second claim, let a linearised observation model be given by
\begin{align}
  r(t_k) \mid \vec{U}(t_k), \vec{\xi}(t_k) & \sim \delta \left(
  H
  \begin{pmatrix}
      \vec{U}(t_k) \\
      \vec{\xi}(t_k)
    \end{pmatrix}
  + b
  \right)
\end{align}
for appropriate $H$ and $b$ (which are explained in \cref{sec:white-noise} below).
Now, the distribution $p(r(t_k) \mid r(t_{k-1}))$ is Gaussian and thus fully defined by its mean $\Ebb[r(t_k)]$ and its covariance $\Cbb[r(t_k)]$.
The covariance $\Cbb[r(t_k)]$ is again of the form $\gamma^2 \breve{S}_k$ for some $\breve{S}_k$ because every filtering covariance is, and the information operator is noise-free.
Due to the prediction error decomposition \citep{schweppe1965evaluation},
\begin{align}
  p(r(t_0), ..., r(t_K) \mid \gamma) = p(r(t_0) \mid \gamma) \prod_{k=1}^N p(r(t_k) \mid r(t_{k-1}), \gamma).
\end{align}
Recall that the dimension of $r(t_k)$ is $N+1$ because $\Xbb$ consists of $N+1$ grid points.
Because everything is Gaussian, the negative log-likelihood (as a function of $x$) decomposes into
\begin{align}
  - \log p(x; r(t_0), ..., r(t_K) \mid \gamma) = \frac{1}{2} \left(\sum_{k=0}^K \|x - \Ebb[r(t_k)]\|_{\Cbb[r(t_k)]^{-1}}^2 - {(K+1)(N+1)} \log \gamma^2 \right).
\end{align}
Setting the $\gamma$-derivative of the negative log-likelihood to zero, i.e.\ maximising it with respect to $\gamma$, yields the MLE from \cref{eq:quasi_mle_output_scale}.
Overall, the derivation is very similar to those provided by \citet{tronarp2019probabilistic,bosch2021calibrated} for ODE initial value problems, and \citet{kramer2021linear} for ODE boundary value problems.

\section{FULLY NONLINEAR (WHITE-NOISE) APPROXIMATE INFERENCE}
\label{sec:white-noise}
The main paper explained white-noise approximations only using the example of semilinear PDEs.

This section recalls the basic concepts for linear PDEs and the more general case of fully nonlinear PDEs.
The latter goes together with an explanation of Taylor-series linearisation of the nonlinear PDE vector field, so it also supports the statements in \cref{subsec:inference}.
The central idea is to regard the PDE vector field $F(t,x,u(t,x), \Dcal u(t,x))$ as a function of $U$ and $\xi$, instead of $u$ and $\Dcal u$, and linearise each non-linearity with a first order Taylor series.

\subsection{Linear PDEs}
At first, we explain the general idea for a linear PDE
\begin{align}
  F(t,x,u(t,x), \Dcal u(t,x)) = A u(t,x) + B \Dcal u(t,x)
\end{align}
for some coefficients $A$ and $B$.
If $u$ is scalar-valued, $A$ and $B$ are scalars.
If $u$ is vector-valued, $A$ and $B$ are matrices.
$\Dcal u = D u + \xi$ allows for recasting the vector field $F$ as a function of $u$ and $\xi$, instead of $u$ and $\Dcal u$,
\begin{align}
  \tilde F(t, U(t), \xi(t)) := F(t, \Xbb, U(t), \xi(t))= A U(t) + B D U(t) + B \xi(t).
\end{align}
Since this equation is linear in $u$ and $\Dcal u$, thus also in $U$ and $\xi$, there is no need for Taylor-series approximation.
The linear measurement model emerges as
\begin{align}\label{eq:linear-pde-residual}
  r_\text{linear}(t_k) \mid \vec{U}(t_k), \vec{\xi}(t_k) \sim \delta \left[ \dot U(t_k) - A U(t_k) - B D U(t_k) - B \xi(t_k) \right].
\end{align}
If the temporal evolution of $\xi$ is ignored through replacing $k_t$ with a white-noise approximation as in \cref{sec:white-noise-approximation}, \cref{eq:linear-pde-residual} becomes
\begin{align}
  r_\text{linear}(t_k) \mid \vec{U}(t_k) \sim \Ncal \left( \dot U(t_k) - A U(t_k) - B D U(t_k), B E B^\top\right).
\end{align}
The same principle can be generalised to fully nonlinear problems as follows.

\subsection{Nonlinear PDEs}
Next, consider a fully nonlinear PDE vector field $F(t,x,u(t,x), \Dcal u(t,x))$.
This general setting includes semilinear and quasilinear systems of equations.
As before, we discretise the spatial domain using $\Xbb$, use $\Dcal u = D U + \xi$, rewrite
\begin{align}
  \tilde F(t,U(t),\xi(t)) := F(t,\Xbb,u(t,\Xbb), D u(t,\Xbb) + \xi(t)),
\end{align}
and infer the solution via $\tilde F$.
Since $\tilde F$ is nonlinear, we have to linearise it before proceeding.
Let $\nabla_U \tilde F$ be the derivative of $\tilde F$ with respect to $U$, and $\nabla_\xi \tilde F$ be its $\xi$-counterpart.
We linearise $\tilde F$ at some $\eta = (\eta_U, \eta_\xi) \in \Rbb^{2(N+1)}$,
\begin{align}
  \tilde F(t, U(t), \xi(t))
   & \approx \tilde F(t, \eta_U, \eta_\xi) + \nabla_U \tilde F (t, \eta_U, \eta_\xi)(U(t) - \eta_U)
  + \nabla_\xi \tilde F (t, \eta_U, \eta_\xi)(\xi(t) - \eta_\xi )                                   \\
   & =: H_U U(t) + H_\xi \xi(t) + b(t).
\end{align}
$\eta_U$ and $\eta_\xi$ are commonly chosen as the predicted mean of $\xi$ and $U$, which corresponds to extended Kalman filtering \citep{sarkka2019applied}.
As a result of the linearisation, the information model reads
\begin{align}\label{eq:linearised-pde-information-latent}
  r_\text{nonlinear}(t_k) \mid \vec{U}(t_k), \vec{\xi}(t_k) \sim \delta \left[ \dot U(t_k) - H_U U(t_k) - H_\xi \xi(t_k) - b \right],
\end{align}
or, for the white-noise version,
\begin{align}\label{eq:linearised-pde-information-white}
  r_\text{nonlinear}(t_k) \mid \vec{U}(t_k) \sim \Ncal \left( \dot U(t_k) - H_U U(t_k) - b,  H_\xi E H_\xi^\top  \right).
\end{align}
\cref{eq:linearised-pde-information-white,eq:linearised-pde-information-latent} are both linear in the states, and inference is possible with Kalman filtering and smoothing.

\end{document}